\documentclass[11pt]{article}
\usepackage[english]{babel}
\usepackage{amsmath}
\usepackage{amsthm}
\usepackage{amssymb}
\usepackage[dvips]{color}
\newtheorem{theorem}{Theorem}[section]
\newtheorem{definition}[theorem]{Definition}

\newtheorem{proposition}[theorem]{Proposition}

\newtheorem{remark}[theorem]{Remark}
\newtheorem{example}[theorem]{Example}

\newcommand{\Er}{\mathcal{M}}

\newcommand{\rr}{\mathbb{R}}
\newcommand{\cc}{\mathbb{C}}

\newcommand{\pp}{\partial}

\newcommand{\unx}{\underline{x}}

\newcommand{\uns}{\underline{s}}
\newcommand{\unp}{\underline{p}}
\newcommand{\una}{\underline{a}}

\title{\bf A new functional calculus for noncommuting operators}

\author{Fabrizio Colombo\\Dipartimento di Matematica\\Politecnico di
Milano\\Via Bonardi, 9\\20133 Milano,
Italy\\fabrizio.colombo@polimi.it
\\
\and Irene Sabadini\\Dipartimento di Matematica\\Politecnico di
Milano\\Via Bonardi, 9\\20133 Milano,
Italy\\irene.sabadini@polimi.it \and Daniele C.
Struppa\\Department of Mathematics \\and Computer Sciences
\\Chapman University\\Orange, CA 92866 USA,
\\struppa@chapman.edu}

\date{ }
\begin{document}
\maketitle
\begin{abstract}
In this paper we use the notion of slice monogenic functions
\cite{slicecss} to define a new functional calculus for an
$n$-tuple $T$ of not necessarily commuting operators. This
calculus is different from the one discussed in \cite{jefferies}
and it allows the explicit construction of the eigenvalue equation
for the $n$-tuple $T$ based on a new notion of spectrum for $T$.
Our functional calculus is consistent with the Riesz-Dunford
calculus in the case of a single operator.
\end{abstract}



\section{Introduction}

In a series of interesting papers, see e.g. \cite{jmc},
\cite{jmcpw}, \cite{mcp} as well as in the book \cite{jefferies} and the references
therein, the authors have developed a monogenic functional calculus,
whose purpose is to deal with $n$-tuples of not necessarily commuting
operators. The interest in this problem can be traced to the early
works of Taylor, see for
example \cite{Taylor}, \cite{Taylor3}, and it is of great physical interest.
In this context, one has other approaches such as the Weyl calculus.
A complete review of these theories can be found in \cite{jefferies} to which we
refer the reader interested in the physical origin of the problem, in a
variety of different
approaches, and in their mutual relationships.

No matter how a functional calculus is developed, one will naturally
require that it be consistent with the case in which only one operator is
considered, and that it be sufficiently flexible to handle both
commuting and noncommuting operators. In the case of a single
operator, one expects to find the usual Riesz-Dunford calculus.
When several operators are considered, it is natural to look for
functions defined on $\rr^n$ with values in noncommuting algebras
which may allow the formal treatment of non commutativity. This
set up  is naturally within the scope of monogenic functions
\cite{bds}.

Let $\rr_n$ be the real Clifford algebra, i.e. the real algebra generated by
the $n$
units $e_1,\dots ,e_n$ such that $e_ie_j+e_je_i=-2\delta_{ij}$. An
element $(x_0,x_1,\ldots,x_n)\in\rr^{n+1}$ can be naturally identified with
the element $x_0+x_1e_1+\ldots+x_ne_n\in\rr_n$. A differentiable function $f:\
U\subseteq \rr^{n+1}\to\rr_n$ is said to be monogenic if it is in
the kernel of the Dirac operator $\pp_{x_0}+e_1\pp_{x_1}+\dots
+e_n\pp_{x_n}$, where $\pp_{x_i}$ is shorthand for $\pp/\pp_{x_i}$.
The theory of such functions is fully developed in
\cite{bds} (for the case of several Dirac operators see \cite{csss}),
and the properties that such functions enjoy closely
resemble those of holomorphic functions of a single complex
variable. In particular, they can be represented by means of a
Cauchy kernel. Specifically, let $\Sigma_n$ denote the volume
of the unit $n$-sphere in
$\rr^{n+1}$, let $\omega$, $x\in \rr^{n+1}$ and $\omega\not=x$,
then if
$$
G(\omega, x)=\frac{1}{\Sigma_n}\frac{\bar\omega -\bar x}{|\omega-x|^{n+1}}
$$
where $\bar x:=x_0-x_1e_1-\ldots-x_ne_n$,   we have
$$
\int_{\partial\Omega}G(\omega, x)n(\omega)f(\omega)d\mu(\omega)=
\left\{%
\begin{array}{ll}
    f(x), & {\rm if} \ \ x\in \Omega, \\
    0, & {\rm if} \ \ x\not\in \Omega, \\
\end{array}%
\right.
$$
where  $\Omega \subset\rr^{n+1}$ is a
bounded open set with smooth boundary
$\partial\Omega$ and exterior unit normal $n(\omega)$,
and $\mu$ is the surface measure of $\partial\Omega$.
Note that the kernel $G(\omega, x)=G_\omega(x)$ can be expanded, see \cite{bds},
as
$$
G_\omega(x)=\sum_{k\geq 0} \left( \sum_{(\ell_1,...,\ell_k)} W_{\ell_1,...,\ell_k}(\omega)
V^{\ell_1,...,\ell_k}(x)\right)
$$
in the region $|x| <|\omega|$
where, for each $\omega\in \rr^{n+1}\backslash \{0\}$,
$W_{\ell_1,...,\ell_k}(\omega)=(-1)^{k}
\partial_{\omega_{\ell_1}}...\partial_{\omega_{\ell_k}}G_\omega(0)
$, $V^{\ell_1,...,\ell_k}(x)=\frac{1}{k!}\sum_{j_1,...,j_k}z_{j_1}...z_{j_k}
$, $z_j=x_je_0-x_0e_j$, and the sum is taken over
all different permutation of $\ell_1,...,\ell_k$.

Consider now an $n$-tuple $T=(T_1,...,T_n)$ of bounded linear operators acting
 on a Banach space $X$ and let $R>(1+\sqrt{2})\|\sum_{j=1}^nT_je_j\|$.
If we formally replace $z_j$ by $T_j$ in the Cauchy kernel series, it
can be shown (see \cite{jefferies}, Lemma 4.7) that
$$
G_\omega(T):=\sum_{k\geq 0} \left( \sum_{(\ell_1,...,\ell_k)} W_{\ell_1,...,\ell_k}(\omega)
V^{\ell_1,...,\ell_k}(T)\right)
$$
converges uniformly for all $\omega \in \rr^{n+1}$ such that $|\omega|\geq R$.
This fact leads to the following definition \cite{jefferies}:
{\it
the monogenic spectrum $\tilde\gamma(T)$ of the $n$-tuple
 $T$ is the complement
of the largest open set $U$ in $\rr^{n+1}$ in which the function $G_\omega(T)$
above is the restriction
of a monogenic function with domain $U$.
}

It is possible to show that if
$T$ is an $n$-tuple of noncommuting bounded linear operators satisfying
suitable reality conditions on their joint spectrum (see \cite{jefferies}), then
the function $\omega \to G_\omega(T)$
is the restriction to the region $|\omega| >(1+\sqrt{2})\|\sum_{j=1}^nT_je_j\|$
of a monogenic function defined off $\mathbb{R}^n$.

Denote with the same symbol $G_\omega(T)$ its maximal monogenic extension.
Under these hypotheses, if
$\Omega\subseteq\rr^{n+1}$ is a bounded open neighborhood
 of $\gamma(T)$
with smooth boundary and if
$f$ is a monogenic
function defined in an open neighborhood of $\overline{\Omega}$, then one can show that
the expression
$$
f(T):=\int_{\pp\Omega} G_\omega(T) n(\omega)f(\omega)d\mu(\omega)
$$
is well defined.

In the case of commuting bounded linear operators the spectrum can be determined
in an explicit way in view of the following result \cite{jefferies}:
\begin{theorem}
Let $T=(T_1,...,T_n)$ be a $n$-tuple of commuting bounded linear operator acting
 on a Banach space $X$ and suppose that
the spectrum of $T_j$ is real for all $j=1,...,n$.
Then $\gamma(T)$ is the complement in $\rr^n$ of the set of all $\lambda\in \rr^n$
for which the operator $\sum_{j=1}^n(\lambda_j\mathcal{I}-T_j)^{2}$ is invertible in ${\cal L}(X)$.
\end{theorem}
We observe that the condition of invertibility of
$\sum_{j=1}^n(\lambda_j\mathcal{I}-T_j)^2$
gives an eigenvalue equation which, at least in some cases, can be easily written. If
for example
$n$ is odd, it is possible to write explicitly the Cauchy kernel as
$$
G_\omega(T)=\frac{1}{\Sigma_n}  \left( \omega_0^2\mathcal{I}+\sum_{j=1}^{n}
(\omega_j\mathcal{I}-T_j)^2\right)^{(-n-1)/2}
\overline{( \omega\mathcal{I} -T)},
$$
whose singularities lie on the set
$$
\left\{(0,\omega_1,\ldots ,\omega_n)\in\rr^{n+1}\ | \ 0\in\sigma\left(
\sum_{j=1}^n (\omega_j\mathcal{I}-T_j)^2\right)\right\}.
$$
In the case the operators $T_j$ do not commute the term
$\frac{1}{\Sigma_n}  |\omega\mathcal{I}-T|^{-n-1}\overline{( \omega\mathcal{I} -T)}$
is not the sum of the Cauchy kernel series so that it is much more difficult
to determine the the spectrum.

In this paper we use a different approach to functional calculus.
Specifically we replace the
use of monogenic functions with the new concept, see \cite{slicecss},
of slice-monogenic functions.
These functions, whose  definition we recall in section 2, have the great advantage that
polynomials as well as power
series  are special cases of slice-monogenic functions (this is in contrast to the
usual definition of monogenic functions).
The key observation which will make our approach successful
is the fact that the sum
of the so-called S-resolvent operator series (see (\ref{Sresolv}))
is a function  which can be utilized even in parts of the
space where the series does not
converge. This remark will make it possible to compute the spectrum also when
the operators do not commute.
In fact, the term $-(T^2-2T Re[s]+|s|^2\mathcal{I})^{-1}(T-\overline{s}\mathcal{I})$
is the sum of the series (\ref{Sresolv}) also when the operators $T_j$ do not commute,
and the singularity of the the Cauchy kernel
(the analogue of the maximal extension $G_\omega(T)$) is given by
$(T^2-2T Re[s]+|s|^2\mathcal{I})v=0$.

The outline of the paper is the following: in section 2 we introduce the concept of
slice-monogenic functions and we recall the properties we need to develop
our functional calculus. In section 3 we define the S-resolvent operator and we give
the notion of $S$-spectrum, we show the S-resolvent equation and we develop the functional
calculus for bounded operators. Finally in section 4 we treat the case of unbounded operators.

{\it Acknowledgements} The first and third authors are grateful to Chapman University for the
hospitality during the period in which this paper was written. They are also indebted to
G.N.S.A.G.A. of INdAM and
the Politecnico di Milano for partially supporting their visit. The authors thank
the anonymous referee for the valuable suggestions to improve the paper.

\section{ Slice monogenic functions}
In this section we collect the basic results on the theory of slice monogenic functions,
developed by
the authors in \cite{slicecss}, to which we refer for the missing proofs in this section.
Let $\rr_n$ be the real Clifford algebra over $n$ units $e_1,\dots,e_n$ such that $e_ie_j+e_je_i=-2\delta_{ij}$. An element in the
Clifford algebra will be denoted by $\sum_A e_Ax_A$ where
$A=i_1\ldots i_r$, $i_\ell\in \{1,2,\ldots, n\}$, $i_1<\ldots <i_r$ is a multi-index
and $e_A=e_{i_1} e_{i_2}\ldots e_{i_r}$. An element $\unx\in\rr^n$ can be identified
with a 1-vector in the Clifford algebra: $(x_1,x_2,\ldots,x_n)\mapsto \unx=
x_1e_1+\ldots+x_ne_n$. A function $f:\ U\subseteq \rr^n\to\rr_n$ is seen as
a function $f(\unx)$ of $\unx$.
\par\noindent
An element in $\rr_n^0\oplus\rr_n^1$ will be written as
 $$
 x=x_0+\unx=x_0+ \sum_{j=1}^nx_je_j.
 $$
  In the sequel  the real part  $x_0$ of $x$ will be also denoted by $Re [x]$.
  \par\noindent
Let us denote by $\mathbb{S}$ the sphere of unit 1-vectors in $\mathbb{R}^n$, i.e.
$$
\mathbb{S}=\{ \unx=e_1x_1+\ldots +e_nx_n\ |\  x_1^2+\ldots +x_n^2=1\}.
$$
The complex line $\mathbb{R}+I\mathbb{R}$ passing through $1$ and $I\in \mathbb{S}$ will be denoted by $L_I$.
\par\noindent
An element belonging to $L_I$ will be denoted by $u+Iv$, for $v$, $v\in \mathbb{R}$.
\par\noindent
Observe that $L_I$, for every $I\in \mathbb{S}$, is a real subspace of $\rr^{n+1}$ isomorphic to the complex plane.
\begin{definition} Let $U\subseteq\rr^{n+1}$ be a domain and let
$f:\ U\to\rr_n$ be a real differentiable function. Let
$I\in\mathbb{S}$ and let $f_I$ be the restriction of $f$ to the
complex line $L_I$.
We say that $f$ is a (left) slice-monogenic function if for every
$I\in\mathbb{S}$
$$
\frac{1}{2}\left(\frac{\partial }{\partial u}+I\frac{\partial
}{\partial v}\right)f_I(u+Iv)=0.
$$
\end{definition}
Analogously, it is possible to define a notion of right slice-monogenicity which gives
a theory  equivalent to the one
left  slice-monogenic functions.
In the sequel, unless otherwise stated, we will consider monogenicity on the left and,
for simplicity, sometimes we will denote by
$\overline{\partial}_I$ the operator
$\frac{1}{2}\left(\frac{\partial }{\partial u}+I\frac{\partial
}{\partial v}\right)$ and we will refer to left slice monogenic
functions as s-monogenic functions. We will also introduce a notion of $I$-derivative
by means of the operator
$$
{\partial}_I:=\frac{1}{2}\left(\frac{\partial }{\partial u}-I\frac{\partial
}{\partial v}\right).
$$
\begin{remark}{\rm The s-monogenic functions on $U\subseteq\rr^{n+1}$ form a
right module $\Er (U)$.
}
\end{remark}
\begin{remark}{\rm  For each $a_m\in \mathbb{R}_n$ the monomials $x \mapsto x^m a_m$
are  left s-monogenic, while the monomials $x \mapsto a_m x^m$ are  right s-monogenic.
Thus also polynomials $\sum_{m=0}^N x^ma_m$
are left s-monogenic  and  any power series $\sum_{m=0}^{+\infty} x^ma_m$
is left s-monogenic in
its domain of convergence. The function $R(x)=(x-y_0)^{-m}$, $m\in\mathbb{N}$
is s-monogenic (left and right) if and only if $y_0\in\rr$.}
\end{remark}
\begin{definition}
Let $U$ be a domain in $\rr^{n+1}$ and let $f:\ U\to\rr_n$ be an s-monogenic function.
Its s-derivative $\pp_s$ is defined as
\begin{equation}\label{s-derivative}
\pp_s(f)=\left\{\begin{array}{ll}
\pp_I(f)(x),\quad & x=u+Iv,\ v\not=0,\\
\pp_uf (u),\quad & {x=u\in\rr}.
\end{array}\right.
\end{equation}
\end{definition}
\noindent
Note that the definition of derivative is well posed because it is applied
only to s-monogenic functions.
Furthermore, any holomorphic function
$f:\ \Delta (0,R)\to \cc$ can be extended (uniquely, up to a choice of an order
for the elements in the basis of $\rr_n$) to an s-monogenic function
$\tilde f: B(0,R)\to\rr_n$.

A key fact is that any s-monogenic function can be developed
into power series and also that it admits a Cauchy integral representation.
\begin{proposition}
If $B=B(x_0,R)\subseteq\rr^{n+1}$ is a ball centered in a real point $x_0$ with radius $R>0$, then
$f:\ B\to\rr_n$
is s-monogenic if and only if it has a series expansion of the form
\begin{equation}\label{serie}
f(x)=\sum_{m\geq 0}x^m\   \frac{1}{m!} \frac{\pp^mf}{\pp u^m}(x_0)
\end{equation}
converging on $B$.
\end{proposition}
\noindent
Given an element $x=x_0+\unx\in\rr^{n+1}$ let us set
$$
I_x=\left\{\begin{array}{l}
\displaystyle\frac{\unx}{|\unx|}\quad{\rm if}\ \unx\not=0,\\
{\rm any\ element\ of\ } \mathbb{S}{\rm\ otherwise.}\\
\end{array}
\right.
$$
We have the following:
\begin{theorem} Let $B=B(0,R)\subseteq\rr^{n+1}$ be a ball with center in $0$
and radius $R>0$ and let $f:\ B\to\rr_n$ be an s-monogenic function. If $x\in B$ then
$$
f(x)=\frac{1}{2\pi } \int_{\pp\Delta_x(0,r)} (\zeta -x)^{-1}\,
d\zeta_{I_x}{f(\zeta)}
$$
where $\zeta\in L_{I_x} \bigcap B$, $d\zeta_{I_x}=-d\zeta{I_x}$
and $r>0$ is such that
$$
\overline{\Delta_x(0,r)}=\{u+I_x v\ |\ u^2+v^2\leq r^2\}
$$
contains $x$ and is contained in $B$.
\end{theorem}
\begin{remark}{\rm An analogue statement holds for regular functions in an open ball centered in a real point $x_0$.}
\end{remark}
\begin{remark}{\rm
If $f$ is an s-monogenic function on a domain $U$ and $\gamma:\ [a,b]\to \rr^{n+1}$ is a curve, then the integral
$\int_\gamma f(\xi) d\xi$ is defined as $\int_a^bf(\gamma(t))\ \gamma'(t)\ dt$. In particular, the curve $\gamma$ can have values on a complex plane $L_I$.
}
\end{remark}
The key ingredient to define a functional calculus is what we
call noncommutative Cauchy kernel
series.
\begin{definition}\label{Cauchykernel}
Let $x=Re[x]+\unx$,
$s=Re[s]+\uns$
be such that $sx\not=xs$. We will call noncommutative Cauchy kernel series the following expansion
\begin{equation}\label{cauchykernseries}
S^{-1}(s,x):=\sum_{n\geq 0}x^ns^{-1-n}
\end{equation}
defined for $|x|<|s|$.
\end{definition}

\begin{theorem} (See \cite{slicecss})\label{inversoalgebrico}
Let $x=Re[x]+\unx$,
$s=Re[s]+\uns$
be such that $xs\not= sx$. Then
$$
\sum_{n\geq 0}x^ns^{-1-n}=-(x^2-2xRe[s]+|s|^2)^{-1}(x-\overline{s})
$$
for $|x|<|s|$.
\end{theorem}

We will call  the expression
$(x^2-2xRe[s]+|s|^2)^{-1}(x-\overline{s})$, defined for $x^2-2xRe[s]+|s|^2\not=0$,
 {\em noncommutative Cauchy kernel}.
 Therefore note that the noncommutative Cauchy kernel
is defined on a  set which is larger then the set $\{(x,s)\ : \ |x|<|s|\}$ where
the noncommutative Cauchy kernel series is defined.
Since $x=\overline{s}$ is a solution of $\overline{s}^2-2\overline{s}Re[s]+|s|^2=0$,
one may wonder if
the factor $(x-\bar s)$ can be simplified from the expression of the noncommutative
Cauchy kernel. However, as shown in
the next result, this is not possible and the noncommutative Cauchy kernel
 cannot be extended to a continuous function in
$x=\overline{s}$.
With an abuse of notation, we will denote
the noncommutative Cauchy kernel series and the noncommutative Cauchy kernel with the same symbol
$S^{-1}(s,x)$.
\begin{theorem}
 Let $S^{-1}(s,x)$ be the noncommutative Cauchy kernel
 with $xs\not= sx$. Then $S^{-1}(s,x)$ is irreducible and
$
\lim_{x\to \overline{s}}S^{-1}(s,x)
$
does not exist.
\end{theorem}
\begin{proof}
We prove that we cannot find a degree one polynomial $Q(x)$ such
that
$$
x^2-2xRe[s]+|s|^2=(s+x-2 \ Re [s])Q(x).
$$
The existence of $Q(x)$ would allow the simplification
$$
S^{-1}(s,x)=Q^{-1}(x)(s+x-2 \ Re [s])^{-1}(s+x-2 \ Re
[s])=Q^{-1}(x).
$$
We proceed as follows: first of all note that $Q(x)$ has to be a
monic polynomial of degree one, so we set
$$
Q(x)=x-r
$$
where   $r=r_0+\sum_{j=1}^nr_je_j$. The equality
$$
(s+x-2 \ Re [s])(x-r)=x^2-2xRe[s]+|s|^2
$$
gives
$$
sx-sr-xr +2 r\ Re [s] -|s|^2=0.
$$
Solving for $r$, we get
$$
r=(s+x -2 \ Re [s])^{-1}(sx-|s|^2),
$$
which depends on $x$. Let us now prove that the limit does not
exists. Let $\varepsilon =\varepsilon_0+\sum_{j=1}^n\varepsilon_je_j$,
  and consider
$$
S^{-1}(s,\overline{s}+\varepsilon)=
((\overline{s}+\varepsilon)^2-2(\overline{s}+\varepsilon)Re[s]+|s|^2)^{-1}\varepsilon
=((\overline{s}+\varepsilon)^2-2(\overline{s}+\varepsilon)Re[s]+|s|^2)^{-1}\varepsilon
$$
$$
=(\overline{s}\varepsilon+\varepsilon\overline{s} +\varepsilon^2
-2\varepsilon Re[s])^{-1}\varepsilon
$$
$$
=(\varepsilon^{-1}(\overline{s}\varepsilon+\varepsilon\overline{s}
+\varepsilon^2 -2\varepsilon Re[s]))^{-1}
=(\varepsilon^{-1}\overline{s}\varepsilon+\overline{s}
+\varepsilon -2 Re[s]))^{-1}.
$$
If we now let $\varepsilon\to 0$, we obtain that the term
$\varepsilon^{-1}\overline{s}\varepsilon$ does not have a limit
because
$$
\varepsilon^{-1}\overline{s}\varepsilon=
\frac{\overline\varepsilon}{|\varepsilon|^2}\overline{s}\varepsilon
$$
contains scalar addends of type
$\displaystyle\frac{\varepsilon_i\varepsilon_j
s_\ell}{|\varepsilon|^2}$ with $i,j,\ell\in \{0,1,2,3\}$ that do
not have limit.
\end{proof}
\noindent
The following result will be useful in the sequel and
its proof follows by a simple computation.
\begin{proposition}\label{equaurea}
Let $s=Re[s]+\uns$.
Then the following identity holds
$$
s^2-2sRe[s]+|s|^2=0.
$$
\end{proposition}

\section{Slice-monogenic functional calculus for bounded operators}

In the sequel, we will consider a Banach space $V$ over
$\mathbb{R}$ (the case of complex Banach spaces can be discussed in a  similar fashion)
 with norm $\|\cdot \|$.
It is possible to endow $V$
with an operation of multiplication by elements of $\rr_n$ which gives
a two-sided module over $\rr_n$.
We recall that
a two-sided module $V$ over $\rr_n$ is called a Banach module over $\rr_n$,
 if there exists a constant $C \geq 1$  such
that $\|va|\leq C\| v\| |a|$ and $\|av|\leq C |a|\| v\|$ for all
$v\in V$ and $a\in\rr_n$.

In the sequel, we will make use of the following notations.
\begin{itemize}
\item By $V$ we denote a Banach space  over
$\mathbb{R}$  with norm $\|\cdot \|$.
\item By $V_n$ we denote the
 two-sided Banach module  over $\rr_n$ corresponding to $V\otimes \rr_n$.
 An element in $V_n$ is of the type $\sum_A v_A\otimes e_A$ (where
 $A=i_1\ldots i_r$, $i_\ell\in \{1,2,\ldots, n\}$, $i_1<\ldots <i_r$ is a multi-index).
The multiplications of an element $v\in V_n$ with a scalar
$a\in \rr_n$ are defined as $va=\sum_A v_A \otimes (e_A a)$ and $av=\sum_A v_A \otimes (ae_A )$.
We will write
$\sum_A v_A e_A$ instead of $\sum_A v_A \otimes e_A$. We define $\| v\|^2_{V_n}=
\sum_A\| v_A\|^2_V$.
\item
$\mathcal{B}(V)$  is the space
of bounded $\mathbb{R}$-homomorphisms of the Banach space $V$ to itself
 endowed with the natural norm denoted by $\|\cdot\|_{\mathcal{B}(V)}$.


\item Let $T_A\in \mathcal{B}(V)$. We define an operator $T=\sum_A T_Ae_A$ and
its action on $v=\sum v_Be_B\in V_n$ as $T(v)=\sum_{A,B}
T_A(v_B)e_Ae_B$. The operator $T$ is a right-module homomorphism which is a bounded linear
map on $V_n$.
The set of all such bounded operators is denoted by $\mathcal{B}_n(V_n)$.
We define $\|T\|^2_{\mathcal{B}_n(V_n)}=\sum_A \|T_A\|^2_{\mathcal{B}(V)}$.
\end{itemize}

\subsection{The $S$-resolvent operator for bounded operators }
Throughout the rest of this section, and unless otherwise specified,
we will only consider operators of the form
$T=T_0+\sum_{j=1}^ne_jT_j$ where $T_\mu\in\mathcal{B}(V)$ for $\mu=0,1,...,n$.
The set of such operators in ${\mathcal{B}_n(V_n)}$ will be denoted by $\mathcal{B}^{\small 0,1}_n(V_n)$.

\begin{definition}
Let $T\in\mathcal{B}^{\small 0,1}_n(V_n)$ and $s=Re[s]+\uns$. We define the
$S$-resolvent operator series as
\begin{equation}\label{Sresolv}
S^{-1}(s,T):=\sum_{n\geq 0} T^n s^{-1-n}
\end{equation}
for $\|T\|< |s|$.
\end{definition}

\begin{theorem}\label{thsmeno}
Let $T\in\mathcal{B}^{\small 0,1}_n(V_n)$ and $s=Re[s]+\uns$.  Then
\begin{equation}\label{ciao}
\sum_{n\geq 0} T^n s^{-1-n}=-(T^2-2T
Re[s]+|s|^2\mathcal{I})^{-1}(T-\overline{s}\mathcal{I}),
\end{equation}
for $\|T\|< |s|$.
\end{theorem}
\begin{proof}
In Theorem \ref{inversoalgebrico} the components
of $x$ and $s$ are real numbers and therefore they obviously commute.
When we formally replace $x$ by operator $T$ we cannot assume that $T_\mu T_\nu=T_\nu T_\mu$
and so we need to verify independently that (\ref{ciao}) still holds.
To this aim, we check that
$-(T-\overline{s}\mathcal{I})^{-1}(T^2-2TRe[s]+|s|^2\mathcal{I})$
is the inverse of $\sum_{n\geq 0} T^n s^{-1-n}$.
In what follows, we assume the convergence of the series to be in the
norm of $\mathcal{B}_n(V_n)$:
$$
-(T-\overline{s}\mathcal{I})^{-1}(T^2-2T
Re[s]+|s|^2\mathcal{I})\sum_{n\geq 0} T^n
s^{-1-n}=\mathcal{I}
$$
so then we get
$$
(-|s|^2\mathcal{I}-T^2+2TRe[s])\sum_{n\geq 0} T^n s^{-1-n}=T+(s-2
\ Re [s])\mathcal{I}.
$$
Observing that $-|s|^2\mathcal{I}-T^2+2TRe[s]$ commutes with $T^n$
we can write
$$
\sum_{n\geq 0} T^n(-|s|^2-T^2+2TRe[s]) s^{-1-n}=T+(s-2 \ Re
[s])\mathcal{I}.
$$
Now expand the series as
$$
\sum_{n\geq 0} T^n(-|s|^2\mathcal{I}-T^2+2Re[s] T) s^{-1-n}=
(-|s|^2\mathcal{I}-T^2+2T Re[s]  ) s^{-1}
$$
$$
+T^1(-|s|^2\mathcal{I}-T^2+2TRe[s]) s^{-2}
+T^2(-|s|^2\mathcal{I}-T^2+2TRe[s]) s^{-3} +...
$$
$$
=-\Big(|s|^2s^{-1}+T(-2s Re[s] +|s|^2)s^{-2} +T^2(s^2-2s
Re[s]+|s|^2)s^{-3} +...\Big)
$$
and using Proposition \ref{equaurea},
we get
$$
\sum_{n\geq 0} T^n(-|s|^2-T^2+2TRe[s]) s^{-1-n}
=-|s|^2s^{-1}\mathcal{I}+Ts^2s^{-2}=-|s|^2s^{-1}\mathcal{I}+T
$$
$$
=
-\overline{s}ss^{-1}\mathcal{I}+T=-\overline{s}\mathcal{I}+T=(s-2
\ Re [s])\mathcal{I}+T.
$$
\end{proof}

\begin{proposition}\label{44}
When $Ts\mathcal{I}=sT$, the operator $S^{-1}(s,T)$ equals $(s\mathcal{I}-T)^{-1}$
when the series (\ref{Sresolv}) converges.
\end{proposition}
\begin{proof} It follows by direct computation.
\end{proof}

\begin{definition}(The $S$-spectrum and the $S$-resolvent set)
Let
$T\in\mathcal{B}^{\small 0,1}_n(V_n)$
 and $ s=Re[s]+\uns$.
We define the $S$-spectrum $\sigma_S(T)$ of $T$ as:
$$
\sigma_S(T)=\{ s\in \mathbb{R}^{n+1}\ \ :\ \ T^2-2 \ Re [s]T+|s|^2\mathcal{I}\ \ \
{\it is\ not\  invertible}\}.
$$
An element in $\sigma_S(T)$ will be called an $S$-eigenvalue.

\noindent
The $S$-resolvent set $\rho_S(T)$ is defined by
$$
\rho_S(T)=\rr^{n+1}\setminus\sigma_S(T).
$$
\end{definition}

\begin{definition}(The $S$-resolvent operator)
Let
$T\in\mathcal{B}^{\small 0,1}_n(V_n)$
 and $ s=Re[s]+\uns \in \rho_S(T)$.
We define the $S$-resolvent operator as
\begin{equation}\label{Sresolvoperator}
S^{-1}(s,T):=-(T^2-2Re[s] T+|s|^2\mathcal{I})^{-1}(T-\overline{s}\mathcal{I}).
\end{equation}
\end{definition}

\begin{example}
{\rm (Pauli matrices)} {\rm As an example, we compute the $S$-spectrum of two Pauli matrices $\sigma_3$, $\sigma_1$
(compare with example 4.10 in \cite{jefferies}):
$$
\sigma_3=\begin{bmatrix}
  1 & 0 \\
  0 & -1 \\
\end{bmatrix}
\quad\quad
\sigma_1=\begin{bmatrix}
  0 & 1 \\
  1 & 0 \\
\end{bmatrix}.
$$
Let us consider the matrix $T=\sigma_3 e_1+\sigma_1
e_2$ and let us compute $T^2-2 Re[s]T+|s|^2\mathcal{I}$.  We obtain the matrix
$$
\begin{bmatrix}
  |s|^2-2 -2 Re[s]e_1 & 2(e_1-Re[s])e_2 \\
-2(e_1+Re[s])e_2 &  |s|^2-2 +2 Re[s]e_1 \\
\end{bmatrix}
$$
whose $S$-spectrum is
$\sigma_S(T)=\{0\}\cup\{s\in\rr^3\ :\ Re[s]=0, |s|=2\}$.
}

\end{example}

\begin{theorem}
Let $T\in\mathcal{B}^{\small 0,1}_n(V_n)$  and $s=Re[s]+\uns\in \rho_S(T)$. Let $
S^{-1}(s,T) $ be the $S$-resolvent operator defined in
(\ref{Sresolvoperator}). Then
$S^{-1}(s,T)$ satisfies the ($S$-resolvent) equation
\begin{equation}\label{Sresolveqbound}
S^{-1}(s,T)s-TS^{-1}(s,T)=\mathcal{I}.
\end{equation}
\end{theorem}
\begin{proof}
Replacing (\ref{Sresolvoperator})
in the above equation we have
\begin{equation}\label{serve21}
-(T^2-2Re[s] T+|s|^2\mathcal{I})^{-1}(T-\overline{s}\mathcal{I})s
+T
(T^2-2Re[s] T+|s|^2\mathcal{I})^{-1}(T-\overline{s}\mathcal{I})=\mathcal{I}
\end{equation}
and applying $(T^2-2Re[s] T+|s|^2\mathcal{I})$ to both hands sides of (\ref{serve21}),
we get
$$
-(T-\overline{s}\mathcal{I})s+(T^2-2Re[s] T+|s|^2\mathcal{I})T
(T^2-2Re[s] T+|s|^2\mathcal{I})^{-1}(T-\overline{s}\mathcal{I})
$$
$$
=T^2-2Re[s] T+|s|^2\mathcal{I}.
$$
Since $T$ and $T^2-2Re[s] T+|s|^2\mathcal{I}$ commute, we obtain the identity
$$
-(T-\overline{s}\mathcal{I})s+T
(T-\overline{s}\mathcal{I})=T^2-2Re[s] T+|s|^2\mathcal{I}
$$
which proves the statement.

\end{proof}

\subsection{Properties of the spectrum and the functional calculus  }

\begin{theorem}\label{strutturaS} (Structure of the $S$-spectrum)
\par\noindent
Let $T\in\mathcal{B}^{\small 0,1}_n(V_n)$
and let $p=Re[p]+\unp$  be an $S$-eigenvalue of $T$
 with $ \unp\neq 0$.
Then all the elements of the sphere $s=Re[s]+\uns$ with
$s_0=p_0$ and $|\uns|=|\unp|$  are $S$-eigenvalues of $T$.
\end{theorem}
\begin{proof}
It is immediate and is left to the reader.
\end{proof}
\begin{definition}
Let $S^n=\{x \in \mathbb{R}^{n+1} : |x|=1\}$ denote the unit sphere of $\mathbb{R}^{n+1}$. For any set $A\subseteq \mathbb{R}^{n+1}$, let us define the {\it circularization of} $A$ as  set 
\begin{displaymath}
circ(A):=\bigcup_{u+vI \in A} u+vS^n.
\end{displaymath}
\end{definition}
\begin{definition}\label{def3.9}
Let $T=T_0+\sum_{j=1}^ne_jT_j\in\mathcal{B}^{\small 0,1}_n(V_n)$.
Let $U \subset \rr^{n+1}$ be an open set such that
\begin{itemize}
\item[(i)] $\pp (U\cap L_I)$ is union of a finite number of
rectifiable Jordan curves  for every $I\in\mathbb{S}$,
\item[(ii) ] $U$ contains the circularization of the $S$-spectrum $\sigma_S(T)$.
\end{itemize}
A function  $f$
is said to be locally s-monogenic on  $\sigma_S(T)$ if there exists
an open set  $U \subset \rr^{n+1}$, as above,
on which $f$ is s-monogenic.
\par\noindent
We will denote
by $\mathcal{M}_{\sigma_S(T)}$ the set of locally s-monogenic functions
on $\sigma_S (T)$.
\end{definition}
\begin{remark}{\rm Note that any open set $U$ containing the circularization of the $S$-spectrum contains open balls with center in $x_0$ for
all $x_0\in\sigma_S(T)\cap \rr$. Moreover, by Theorem \ref{strutturaS},  if the $(n-1)$-sphere $\sigma=\{s\in\rr^{n+1}\ : Re[s]=s_0, |\underline{s}|=r\}$ belongs to $\sigma_S(T)$, then $U$ must contain an open annular domain with center in $s_0\in\rr$. In fact, set $m=\min_{s\in circ(\sigma )}{\rm dist}(s,\pp U)$. Then for any  $R<m$  the annular domain $\{x\in\rr^{n+1}\ |\ r-R<|x -s_0|<r+R\}$ is contained in $U$.}
\end{remark}
\begin{theorem}\label{indipdaui} Let $T\in\mathcal{B}^{\small 0,1}_n(V_n)$
 and  $f\in \mathcal{M}_{\sigma_S(T)}$.
  Let $U\subset \mathbb{R}^{n+1}$ be an open set  as in Definition \ref{def3.9}  and let $U_I=U\cap L_I$ for $I\in \mathbb{S}$.  Then the integral
\begin{equation}\label{integ311}
{{1}\over{2\pi }} \int_{\partial U_I } S^{-1} (s,T)\  ds_I \ f(s)
\end{equation}
does not depend on the choice of the imaginary unit $I$ and on the
open set $U$.
\end{theorem}
\begin{proof}
We first note that the integral (\ref{integ311}) does
not depend on the choice of $U$ by the Cauchy theorem applied on the plane $L_I$, see \cite{slicecss}. We now show the independence of the choice of $I\in\mathbb{S}$. Note that since the $S$-spectrum is bounded (because it is contained in the ball $\{s\in\rr^{n+1}\, :\, |s|\leq \|T\|$\}) we can choose a finite number of open balls $B_1,\ldots, B_\nu$ and of open annular domains $A_1,\ldots, A_\mu$, $\nu, \mu \in \mathbb{N}$, containing the $S$-spectrum of $T$. We observe that thanks to the the Cauchy theorem we can write:
\begin{equation}\label{independence}
{{1}\over{2\pi }}\int_{\pp U_I}S^{-1}(s,T)ds_I f(s)
$$
$$
={{1}\over{2\pi }}\sum_{i=1}^{\nu}\int_{\pp (B_i\cap L_I)} S^{-1}(s,T)ds_I f(s) +{{1}\over{2\pi }}\sum_{i=1}^{\mu}\int_{\pp (A_i\cap L_I)}S^{-1}(s,T)ds_I f(s),
\end{equation}
where the right hand side does not depend on the choice of the $B_i$'s and $A_i$'s.
Since $f$ admits series expansion on the $B_i$'s for Taylor theorem and on $A_i$'s by the Laurent theorem (see \cite{slicecss}), we can integrate term by term.
Let us now choose another
imaginary unit $I'\in\mathbb{S}$, $I\not=I'$ and let us write the analogue of (\ref{independence}) on $L_{I'}$.
The $S$-spectrum contains either real
points or, by  Theorem \ref{strutturaS}, $(n-1)$-spheres of the type $\{s\in\rr^{n+1}\ : Re[s]=s_0, |\underline{s}|=r\}$. Every
complex line $L_I=\mathbb{R}+I\mathbb{R}$ contains all the real points
belonging to the $S$-spectrum.  Let $\{s\in\rr^{n+1}\ : s=s_0+It, |t|=r\}$
be an $(n-1)$-sphere in the $S$-spectrum.
The two points of the sphere lying on the complex line $L_I$ are $s_0\pm rI$
so, on the plane, they have
coordinates $(s_0,\pm r)$. The coordinates of the two intersection points
on a different
complex line $L_{I'}$  are still
$(s_0,\pm r)$, so the right hand side of (\ref{independence}) does not depend on the choice of $I\in\mathbb{S}$.
Thus
$$
{{1}\over{2\pi }}\int_{\partial U_I} S^{-1}(s,T)\
ds_I f(s)= {{1}\over{2\pi }}\int_{\partial U_{I'}} S^{-1}(s,T)\  ds_{I'} f(s).
$$

\end{proof}

We give a result that motivates the functional
calculus.
\begin{theorem}\label{polyaa}
Let $x=Re[x]+\unx$, $a=Re[a]+\una \in \mathbb{R}^{n+1}$, $m\in \mathbb{N}$   and
consider the monomial  $ x^m a$.
Consider $T\in\mathcal{B}^{\small 0,1}_n(V_n)$,
let  $U\subset \mathbb{R}^{n+1}$ be an open set  as in Definition \ref{def3.9},
and set $U_I=U\cap L_I$ for $I\in\mathbb{S}$.
Then
\begin{equation}\label{TA}
T^m a= {{1}\over{2\pi }} \int_{\partial U_I } S^{-1} (s,T)\  ds_I \
s^m \ a.
\end{equation}
\end{theorem}
\begin{proof}
Let us consider the power series expansion for the operator
$S^{-1} (s,T)$ and a circle $C_r$ centered in the origin and of radius $r> \|T\|$.
We have:
\begin{equation}
{{1}\over{2\pi }}\int_{\pp U_I} S^{-1} (s,T)\  ds_I \  s^m \ a
={{1}\over{2\pi }}\sum_{n\geq 0} T^n\int_{C_r} s^{-1-n+m}\
ds_I a
 = T^m \ a,
\end{equation}
since
\begin{equation}
\int_{C_r} ds_I  s^{-n-1+m}=0\ \ if \ \ n\not=m, \ \ \ \ \ \
\int_{C_r} ds_I  s^{-n-1+m}=2\pi \ \ if \ \ n=m.
\end{equation}
The Cauchy theorem shows that the above integrals are not affected if
we replace $C_r$ by $\pp U_I$.
\end{proof}
\begin{theorem}\label{compattezaS}(Compactness of $S$-spectrum)
\par\noindent
 Let $T\in\mathcal{B}^{\small 0,1}_n(V_n)$. Then
the $S$-spectrum $\sigma_S (T)$  is a compact nonempty set.
Moreover $\sigma_S (T)$ is
contained in $\{s\in\rr^{n+1}\, :\,  |s|\leq \|T\| \ \}$.
\end{theorem}
\begin{proof}
Let $U\subset \mathbb{R}^{n+1}$ be an open set  as in Definition \ref{def3.9},
and set $U_I=U\cap L_I$ for $I\in\mathbb{S}$.
Then
$$
\frac{1}{2\pi }\int_{\partial U_I} S^{-1}(s,T)\ ds_I
\  s^m\, =T^m.
$$
In particular, for $m=0$, we have
$$
\frac{1}{2\pi }\int_{\partial U_I} S^{-1}(s,T)\
ds_I \  =\mathcal{I},
$$
where $\mathcal{I}$ denotes the identity operator, which shows that
$\sigma_S (T) $ is a nonempty set otherwise the integral would be zero
by the vector valued version of Cauchy's theorem.
 We now show that the $S$-spectrum is bounded.
The series $ \sum_{n\geq 0} T^n s^{-1-n} $ converges if and only
if $\|T\|< |s|$ so the $S$-spectrum is contained in the set $\{s
\in \rr^{n+1}\, :\, |s| \leq \|T\| \}$, which is bounded and closed
because
 the complement of $\sigma_S (T) $
is open. Indeed, the function $g: s\mapsto T^2-2 Re[s] T+|s|^2\mathcal{I}$
is trivially continuous and, by Theorem 10.12 in \cite{rudin}, the set  $\mathcal{U}(V_n)$
of all invertible elements of $\mathcal{B}_n(V_n)$  is an open set in $\mathcal{B}_n(V_n)$.
Therefore $g^{-1}(\mathcal{U}(V_n))=\rho_S(T)$ is an open set in $\rr^{n+1}$.
\end{proof}

\noindent The preceding discussion allows to give the following
definition.
\begin{definition} \label{fdiT}
Let
 $T\in\mathcal{B}^{\small 0,1}_n(V_n)$ and  $f\in \mathcal{M}_{\sigma_S(T)}$.
Let $U\subset \mathbb{R}^{n+1}$ be an open set  as in Definition \ref{def3.9},
and set $U_I=U\cap L_I$ for $I\in\mathbb{S}$.
We define
\begin{equation}\label{FC}
f(T)= {{1}\over{2\pi }} \int_{\partial U_I } S^{-1} (s,T)\  ds_I \
f(s).
\end{equation}
\end{definition}
\begin{remark}{\rm
To compare our new functional calculus with the existing versions which the reader can find
in the literature (see e.g. \cite{jefferies} and its references) we will now consider the subset
$\mathcal{B}^{\small 1}_n(V_n)\subset \mathcal{B}^{\small 0,1}_n(V_n)$ whose elements
are operators of the form $T=\sum_{j=1}^n T_je_j$ where $T_j$ are linear operators acting
on the Banach space $V$. When $n=1$, this corresponds to considering a single operator
$T_1$ and $T=T_1e_1$.
To compute the $S$-spectrum we have to consider the $S$-eigenvalue equation. Since in this case,
the variable $s\in\mathbb{C}$ commutes
with $T$, the $S$-eigenvalue equation reduces to the classical eigenvalue
equation. Finally, since the theory of s-monogenic
functions coincides with the theory of holomorphic functions in one complex variables, our
calculus reduces to the Riesz-Dunford calculus. Let $f(z)$ be any function holomorphic
on the spectrum of $T_1$. The Riesz-Dunford calculus allows to compute $f(T_1)$. In our case,
to get exactly the function $f(T_1)$ we need to consider $\tilde f(z)=f(-ze_1)=f(x_1-e_1 x_0)$ where we
have denoted $z=x_0+e_1x_1$. This is not surprising, since we are considering not the given operator
$T_1$ as in the classical case, but its tensor with the imaginary unit $e_1$.
The two calculi are therefore equivalent up to this identification.}
\end{remark}

More generally, the following result holds.

\begin{theorem}
Let $T\in \mathcal{B}^{\small 1}_n(V_n)$
and $f\in \mathcal{M}_{\sigma_S(T)}$.
The slice-monogenic functional calculus $f\mapsto f(T)$ satisfies the following properties:
\begin{enumerate}
\item it is consistent with the Riesz-Dunford functional calculus when $n=1$;
\item it is a right module homomorphism;
\item $fg(T)=f(T)g(T)$ when $f$, $g$ are represented by power series with real coefficients;
\item it is continuous on the space $\mathcal{B}^{\small 1}_n(V_n)\times \mathcal{M}_{\sigma_S(T)}
\to \mathcal{B}_n(V_n)$.
\end{enumerate}
\end{theorem}
\begin{proof}
\begin{enumerate}
\item This fact has been shown above.
\item It is an immediate and follows by  computations similar to those in Theorem \ref{polyaa}.
\item The product of two power series with real coefficients is an s-monogenic function (see
Proposition 2.13 in \cite{slicecss}). The statement
follows from  the definition of $f(T)$ (see (\ref{FC}))  and the $S$-resolvent equation (\ref{Sresolveqbound}).
\item It follows from  the definition of $f(T)$ (see (\ref{FC})) reasoning as for the Riesz-Dunford case.
\end{enumerate}
\end{proof}
\begin{remark}{\rm
We can also consider operators $T\in\mathcal{B}_n^{0,1}(V_n)$: those are right-module
homomorphisms for which we can prove $2,3,4$ of the preceding theorem.}
\end{remark}

\section{Slice-monogenic functional calculus for unbounded operators}

We now show the development of a functional calculus for unbounded operators based on the calculus
already obtained for bounded operators.
Note that if $T$ is a closed operator, the series $\sum_{n\geq 0}T^ns^{-1-n}$ does not
converge.
To overcome this difficulty,
we observe that the right hand side of formula (\ref{ciao}) contains the inverse
of the operator $T^2-2T Re[s]+|s|^2\mathcal{I}$.
From a heuristical point of view, and for suitable $s\in\rr^{n+1}$, the composition
$(T^2-2T Re[s]+|s|^2\mathcal{I})^{-1}(T-\overline{s}\mathcal{I})$ gives a bounded operator
on suitable function spaces.

\subsection{The $S$-resolvent operator for unbounded operators }

\begin{definition}\label{defschaa}
Let $V$ be a Banach space and  $V_n$ be the
 two-sided Banach module  over $\rr_n$ corresponding to $V\otimes \rr_n$.
Let  $T_\mu : {\cal D}(T_\mu) \subset V\to V$
be linear closed densely defined  operators for $\mu=0,1,...,n$.
Let
\begin{equation}\label{domain}
{\cal D}(T)=\{ v\in V_n\ : \ v=\sum_{B}v_Be_B,\ \ v_B\in \bigcap_{\mu=0}^n{\cal D}(T_\mu)\ \}.
\end{equation}
be the domain of the operator
$$
T=T_0+\sum_{j=1}^ne_jT_j,\qquad T : {\cal D}(T) \subset V_n\to V_n.
$$
Let us assume that
\begin{itemize}
\item[1)]
$\bigcap_{\mu=0}^n{\cal D}(T_\mu)$ is dense in $V_n$,
\item[2)]
$T-\overline{s}\mathcal{I}$  is densely defined in $V_n$,
\item[3)]
${\cal D}(T^2)\subset {\cal D}(T)$  is dense in $V_n$,
\item[4)]
$T^2-2T
Re[s]+|s|^2\mathcal{I}$ is one-to-one with range $V_n$.
\end{itemize}

 The $S$-resolvent operator is defined by
\begin{equation}\label{resochiaa}
S^{-1}(s,T)=-(T^2-2T
Re[s]+|s|^2\mathcal{I})^{-1}(T-\overline{s}\mathcal{I}).
\end{equation}
\end{definition}

\begin{remark}{\rm

We observe that, in principle, it is necessary also the following assumption:
\begin{itemize}
\item[5)] the operator $(T^2-2T Re[s]+|s|^2\mathcal{I})^{-1}(T-\overline{s}\mathcal{I})$ is the restriction to the dense subspace
${\cal D}(T)$ of $V_n$ of a bounded linear operator.
\end{itemize}
However this assumption is automatically fulfilled since it follows from the identity
$$
(T^2-2T Re[s]+|s|^2\mathcal{I})^{-1}(T-\overline{s}\mathcal{I})
=
T(T^2-2T Re[s]+|s|^2\mathcal{I})^{-1}-
(T^2-2T Re[s]+|s|^2\mathcal{I})^{-1}\overline{s}\mathcal{I},
$$
which is a consequence of
$$
(T^2-2T Re[s]+|s|^2\mathcal{I})^{-1}T=T(T^2-2T Re[s]+|s|^2\mathcal{I})^{-1},
$$
that can be easily verified applying on the left to both sides the operator  $T^2-2T Re[s]+|s|^2\mathcal{I}$.
}
\end{remark}

\begin{definition}\label{Sresolandspecset}
Let $T:{\cal D}(T)\to V_n $ be a linear closed densely defined operator as in Definition
\ref{defschaa}.
 We define the $S$-resolvent set of $T$ to be the set
\begin{equation}\label{resosetpoint}
\rho_S(T)=\{ s\in \mathbb{R}^{n+1}  \ such\ that \ S^{-1}(s,T)
\ exists \ and \ it \ is \  in \ \mathcal{B}_n(V_n) \}.
\end{equation}
We define the $S$- spectrum of $T$ as the set
\begin{equation}\label{specsetpoint}
\sigma_S(T)=\rr^{n+1}\setminus \rho_S(T).
\end{equation}
\end{definition}
\begin{theorem}\label{Sresequ}($S$-resolvent operator equation)
Let $T:{\cal D}(T)\to V_n $ be a linear closed densely defined operator.
Let $s\in \rho_S(T)$. Then  $S^{-1}(s,T)$ satisfies the ($S$-resolvent) equation
$$
S^{-1}(s,T)s-TS^{-1}(s,T)=\mathcal{I}.
$$
\end{theorem}
\begin{proof}
It follows by a direct computation replacing the
 $S$- resolvent operator into the $S$-resolvent equation.
 \end{proof}

\begin{theorem}\label{bounSres}
Let $T:{\cal D}(T)\to V_n $ be a linear closed densely defined operator.
Let $s\in \rho_S(T)$.
Then the $S$-resolvent operator
 can be represented by
\begin{equation}\label{chiusoS}
S^{-1}(s,T)=\sum_{n\geq 0}(Re[s]\mathcal{I}-T)^{-n-1}(Re[s]-s)^n
\end{equation}
if and only if
\begin{equation}\label{Imesse}
|\underline{s}|\ \|( Re[s]\mathcal{I}-T)^{-1}\|<1.
\end{equation}
\end{theorem}
\begin{proof}
Observe that
$$
(T^2-2T Re[s]+|s|^2\mathcal{I})^{-1}
=
  \left(\mathcal{I}+|\underline{s}|^2(T- Re[s]\mathcal{I})^{-2}\right)^{-1}(T- Re[s]\mathcal{I})^{-2}
$$
$$
=\sum_{n\geq 0}(-1)^n|\underline{s}|^{2n}(T- Re[s]\mathcal{I})^{-2n}(T- Re[s]\mathcal{I})^{-2},
$$
so we get
$$
S^{-1}(s,T)=
\sum_{n\geq 0}(-1)^{n+1}
|\underline{s}|^{2n}(T- Re[s]\mathcal{I})^{-2n-1}
(\mathcal{I}-(T- Re[s]\mathcal{I})^{-1}\underline{s})
$$
$$
=
\sum_{n\geq 0}( Re[s]\mathcal{I}-T)^{-2n-1}( Re[s]-s)^{2n}+
\sum_{n\geq 0}( Re[s]\mathcal{I}-T)^{-2n-2}( Re[s]-s)^{2n+1}
$$
$$
=
( Re[s]\mathcal{I}-T)^{-1}
\sum_{n\geq 0}( Re[s]\mathcal{I}-T)^{-n}( Re[s]-s)^{n}
$$
which converges in the bounded linear operator space if and only if (\ref{Imesse}) holds.
\end{proof}

\begin{remark}{\rm
The previous result implies that given any element $s\in \rho_S(T)$, then any other
$s'\in\rr^{n+1}$ with $|\underline{s'}|=|\underline{s}|$ belongs to the resolvent set, independently
on its real part $s_0$,
in fact $s'$ satisfies the inequality (\ref{Imesse}).
}
\end{remark}

\subsection{Functional calculus for unbounded operators}

Let $V$ be a Banach space and $T=T_0+\sum_{j=1}^me_jT_j$ where $T_\mu : {\cal D}(T_\mu)\to V$
are linear operators for $\mu=0,1,...,n$.
If  at least one of the $T_j$'s
is an unbounded operator then its resolvent  is not defined at infinity.
It is therefore natural to consider closed operators $T$ for which the resolvent
$S^{-1}(s,T)$ is not defined at infinity and to define the extended spectrum as
$$
\overline{\sigma}_S(T):=\sigma_S(T)\cup \{\infty\}.
$$
Let us consider $\overline{\rr}^{n+1}=\rr^{n+1}\cup\{\infty\}$ endowed
with the natural topology: a set is open if and only if it is union of open
discs $D(x,r)$ with center at points in $x\in\rr^{n+1}$ and radius $r$, for some $r$, and/or
union of sets the form $\{x\in\rr^{n+1}\ |\ |x|>r\}\cup\{\infty\}=D'(\infty,r)\cup\{\infty\}$, for some $r$.
\begin{definition}
We say that  $f$ is  s-monogenic function at $\infty$ if $f(x)$ is
an s-monogenic function in a set $D'(\infty,r)$ and
$\lim_{x\to\infty}f(x)$ exists and it is finite. We define $f(\infty)$ to be the
value of this limit.
\end{definition}

\begin{definition}\label{def3.9seconda}
Let $T :{\cal D}(T)\to V_n$ be a linear closed operator as in Definition \ref{defschaa} .
Let  $U \subset \rr^{n+1}$ be an open set such that
\begin{itemize}
\item[(i)] $\pp (U\cap L_I)$ is union of a finite number of rectifiable Jordan curves for every $I\in\mathbb{S}$,
\item[(ii)] $U$ contains the $S$-spectrum $\sigma_S(T)$.
\end{itemize}
A function  $f$  is said to be locally s-monogenic on  $\overline{\sigma}_S(T)$
if it is s-monogenic an open set  $U \subset \rr^{n+1}$ as above and at infinity.
\par\noindent
We will denote
by $\mathcal{M}_{\overline{\sigma}_S(T)}$ the set of locally s-monogenic functions
on $\overline{\sigma}_S(T)$.
\end{definition}
Consider $\alpha\in\rr^{n+1}$ and the homeomorphism
$$
\Phi :\overline{\rr}^{n+1}\to \overline{\rr}^{n+1}  \ \ \ {\rm for}\ \ \ \alpha\in \rr^{n+1}
$$
 defined by
$$
 p=\Phi(s)=(s-\alpha)^{-1}, \ \ \Phi(\infty)=0,\ \ \ \Phi(\alpha)=\infty.
$$

\begin{definition}
Let $T:{\cal D}(T)\to V_n$ be a linear closed operator as in Definition \ref{defschaa} with
$\rho_S(T)\cap \mathbb{R}\neq\emptyset$ and suppose that  $f\in \mathcal{M}_{\overline{\sigma}_S(T)}$.
Let us consider
$$
\phi(p):=f(\Phi^{-1}(p))
$$
and the operator
$$
A:=(T-k\mathcal{I})^{-1},\ \ {\it for\ some}\ \  k\in \rho_S(T)\cap \mathbb{R}.
$$
We define
\begin{equation}\label{fdit}
f(T)=\phi(A).
\end{equation}

\end{definition}

\begin{remark}{\rm Observe that, if $\alpha=k\in\rr$, we have that:

i) the function $\phi$ is s-monogenic  because it is the
composition
of the function $f$ which is s-monogenic and  $\Phi^{-1}(p)=p^{-1}+k$ which is s-monogenic
with real coefficients;

ii) in the case $k\in \rho_S(T)\cap \mathbb{R}$  we  have that
$(T-k\mathcal{I})^{-1}=-S^{-1}(k, T)$.

}
\end{remark}
\begin{theorem}
If $k\in \rho_S(T)\cap\rr \not=\emptyset$ and $\Phi$, $\phi$ are as above, then $\Phi(\overline{\sigma}_S(T))=\sigma_S(A)$ and the relation
$\phi(p):=f(\Phi^{-1}(p))$ determines a one-to-one correspondence  between $f\in \mathcal{M}_{\overline{\sigma}_S(T)}$ and $\phi\in \mathcal{M}_{\overline{\sigma}_S(A)}$.
\end{theorem}
\begin{proof} First we consider the case
 $p\in \sigma_S(A)$ and $p\not= 0$. Recall that
$$
S^{-1}(p,A)=-(A^2-2A Re[p]+|p|^2\mathcal{I})^{-1}(A-\overline{p}\mathcal{I}),
$$
from which we obtain
$$
(A^2-2A Re[p]+|p|^2\mathcal{I})S^{-1}(p,A)=-(A-\overline{p}\mathcal{I}).
$$
Let us apply the operator $A^{-2}$ on the left to get
$$
(\mathcal{I}-2 Re[p] A^{-1}+A^{-2}|p|^2)S^{-1}(p,A)=-(A^{-1}-A^{-2}\overline{p}).
$$
Now we use the relations
$$
A^{-1}=T-k\mathcal{I},\ \ \ A^{-2}=T^2-2kT+k^2\mathcal{I}
$$
to get
$$
(\mathcal{I}-2 Re[p] (T-k\mathcal{I})+(T^2-2kT+k^2\mathcal{I})|p|^2)S^{-1}(p,A)
$$
$$
=-(T-k\mathcal{I}-(T^2-2kT+k^2\mathcal{I})\overline{p}).
$$
Using the identities
\begin{equation}\label{identsk}
s_0|p|^2=k|p|^2+p_0,\ \ \ \ |p|^2|s|^2=k^2|p|^2+2p_0k+1
\end{equation}
we have
$$
(T^2 |p|^2 -2s_0|p|^2T + |s|^2 |p|^2\mathcal{I})
S^{-1}(p,A)=-(T-k\mathcal{I}-(T^2-2kT+k^2\mathcal{I})\overline{p}).
$$
So we get the equalities
$$
S^{-1}(p,A)=-\frac{1}{|p|^2}(T^2  -2s_0T + |s|^2 \mathcal{I})^{-1}
(T-k\mathcal{I}-(T^2-2kT+k^2\mathcal{I})\overline{p})
$$
$$
=-\frac{1}{|p|^2}(T^2  -2s_0T + |s|^2 \mathcal{I})^{-1}
(T{\overline{p}}^{-1}-k{\overline{p}}^{-1}\mathcal{I}-T^2+2kT-k^2\mathcal{I})\overline{p}
$$
$$
=-\frac{1}{|p|^2}(T^2  -2s_0T + |s|^2 \mathcal{I})^{-1}
\Big(-(T^2-2s_0T+|s|^2\mathcal{I})\overline{p}
$$
$$
+( T(2k-2s_0+\overline{p}^{-1})
+(|s|^2-k^2-k\overline{p}^{-1})\mathcal{I})\overline{p}\Big)
$$
$$
=
\frac{\overline{p}}{|p|^2}\mathcal{I}
-\frac{1}{|p|^2}(T^2  -2s_0T + |s|^2 \mathcal{I})^{-1}
( T
+(|s|^2-k^2-k\overline{p}^{-1})(2k-2s_0
+\overline{p}^{-1})^{-1}\mathcal{I})(2k-2s_0+\overline{p}^{-1})\overline{p}.
$$
With some calculation we get
$$
S^{-1}(p,A)=p^{-1}\mathcal{I}
$$
$$
-\frac{1}{|p|^2}(T^2  -2s_0T + |s|^2 \mathcal{I})^{-1}
( T -\overline{s}\mathcal{I}+
[\overline{s}+(|s|^2-k^2-k\overline{p}^{-1})
(2k-2s_0+\overline{p}^{-1})^{-1}] \mathcal{I})(2k-2s_0+\overline{p}^{-1})\overline{p}
$$
and also
$$
S^{-1}(p,A)=
p^{-1}\mathcal{I}
+S^{-1}(s,T)\lambda
-\frac{1}{|p|^2}(T^2  -2s_0T + |s|^2 \mathcal{I})^{-1}
\Lambda\mathcal{I}
$$
where
$$
\lambda:=(2k-2s_0+\overline{p}^{-1})\frac{\overline{p}}{|p|^2}
$$
and
$$
\Lambda:=[\overline{s}+(|s|^2-k^2-k\overline{p}^{-1})
(2k-2s_0+\overline{p}^{-1})^{-1}](2k-2s_0+\overline{p}^{-1})\overline{p}
$$
Using the identities (\ref{identsk}) we have that
$$
\lambda=(2k-2s_0+\overline{p}^{-1})\frac{\overline{p}}{|p|^2}
=
(2k-2s_0+\overline{p}^{-1}){p}^{-1}
$$
$$
=
(2k-2s_0+\overline{s}-k)(s-k)=-(s-k)^2=-{p^{-2}}
$$
and with analogous calculation we get
$$
\Lambda=\overline{s}(2k-2s_0+\overline{p}^{-1})\overline{p}+
(|s|^2-k^2-k\overline{p}^{-1})\overline{p}=0.
$$
So
$$
S^{-1}(p,A)=
p^{-1}\mathcal{I}
-S^{-1}(s,T){p^{-2}},
$$
but also
\begin{equation}\label{importante}
S^{-1}(s,T)=
p\mathcal{I}
-S^{-1}(p,A)p^2.
\end{equation}
So $p\in \rho_S(A)$, $p\not=0$ then $s\in \rho_S(T)$.
\par\noindent
Now take $s\in \rho_S(T)$ and observe that, from the definitions of $S^{-1}(s,T)$ and
of $A$, with analogous calculation as above, we get
$$
S^{-1}(s,T)=-AS^{-1}(p,A)p,
$$
so if  $s\in \rho_S(T)$ then $p\in \rho_S(A)$, $p\not=0$.
\par\noindent
The point $p=0$ belongs to $\sigma_S(A)$
 since $S^{-1}(0,A)=A^{-1}=T-k\mathcal{I}$ is unbounded.
 The last part of the statement is evident from the definition of $\Phi$.
\end{proof}

\begin{theorem}\label{calcfunzformu}
Let $T:{\cal D}(T)\to V_n$ be a linear closed operator as in Definition \ref{defschaa} with
$\rho_S(T)\cap \mathbb{R}\neq\emptyset$ and suppose that  $f\in \mathcal{M}_{\overline{\sigma}_S(T)}$.
Then operator $f(T)$ defined in (\ref{fdit}) is independent of $k\in \rho_S(T)\cap \mathbb{R}$.

Let $W$, be an open set such that $\overline{\sigma}_S(T) \subset W$
and let $f$ be an s-monogenic function on $W\cup\partial W$.
Set $W_I= W\cap L_I$ for $I\in \mathbb{S}$ be such that
its boundary $\partial W_I$ is positively oriented and consists of a finite number
of rectifiable Jordan curves.
Then
\begin{equation}\label{fditfor}
f(T)=f(\infty)\mathcal{I}+\frac{1}{2\pi} \int_{\partial W_I} S^{-1}(s,T)ds_I f(s).
\end{equation}
\end{theorem}
\begin{proof}
The first part of the statement follows from the validity of formula (\ref{fditfor})
since the integral is independent of $k$.
\par\noindent
Given $k\in\rho_S(T)\cap\mathbb{R}$ and the set $W$ we can assume that
$k\not\in W_I\cup\partial W_I$, $\forall I\in\mathbb{S}$
$I\in \mathbb{S}$ since otherwise, by the Cauchy theorem,  we can replace
$W$ by $W'$, on which $f$ is s-monogenic, such that $k\not\in W_I'\cup\partial W_I'$,
without altering the value of the integral (\ref{fditfor}). Moreover, the integral
(\ref{fditfor}) is independent of the choice of $I\in\mathbb{S}$, thanks to the
structure of the spectrum (see Theorem \ref{strutturaS}) and an argument similar to the
one used to prove Theorem \ref{indipdaui}.
\par\noindent
 We have that ${\cal V}_I:=\Phi^{-1}(W_I)$ is an open set that contains
$\sigma_S(T)$ and its boundary $\partial {\cal V}_I=\Phi^{-1}(\partial W_I)$
is positively oriented and consists of a finite number
of rectifiable Jordan curves.  Using the relation (\ref{importante}) we have
$$
  \frac{1}{2\pi} \int_{\partial W_I} S^{-1}(s,T)ds_I f(s)
$$
$$
  = -\frac{1}{2\pi} \int_{\partial {\cal V}_I} \Big(p\mathcal{I}-S^{-1}(p,A)p^2 \Big) p^{-2}dp_I \phi(p)
$$
$$= -\frac{1}{2\pi} \int_{\partial {\cal V}_I} p^{-1}dp_I \phi(p)+\frac{1}{2\pi} \int_{\partial {\cal V}_I} S^{-1}(p,A)dp_I \phi(p)
$$
$$
= -\mathcal{I} \phi(0)+\phi(A)
$$
now by definition $\phi(A)=f(T)$ and $\phi(0)=f(\infty)$ we obtain
$$
\frac{1}{2\pi} \int_{\partial W_I} S^{-1}(s,T)ds_I f(s)=-\mathcal{I} f(\infty)+f(T).
$$

\end{proof}

\end{document}